\documentclass[11pt]{amsart}
\usepackage{amscd} 
\usepackage{amsmath} 
\usepackage{amssymb} 
\usepackage{amsxtra}
\usepackage{xcolor}  
\usepackage{hyperref}
\hypersetup{
     colorlinks,
     linkcolor={red!50!black},
     citecolor={blue!50!black},
     urlcolor={blue!80!black}
   }
\usepackage{cite}
\usepackage[margin=2.5cm]{geometry} 
\usepackage[all]{xy}

\newtheorem{theorem}{Theorem}
\newtheorem{lemma}[theorem]{Lemma}

\theoremstyle{definition} 
\newtheorem{remark}[theorem]{Remark} 
\newtheorem{notation}[theorem]{Notation} 
\newtheorem{definition}[theorem]{Definition} 
\newtheorem*{remark*}{Remark}
\newtheorem*{notation*}{Notation}
\newtheorem*{acknowledgments*}{Acknowledgments} 

\newcommand{\cA}{\mathcal{A}}
\newcommand{\cB}{\mathcal{B}}
\newcommand{\C}{\mathbb{C}}
\newcommand{\N}{\mathbb{N}} 

\newcommand{\cF}{\mathcal{F}}
\newcommand{\cO}{\mathcal{O}}
\newcommand{\fm}{\mathfrak{m}} 
 
\newcommand{\new}{\text{\rm new}}
\newcommand{\barQ}{\widebar{Q}}
\newcommand{\barU}{\widebar{U}}

\makeatletter
\newcommand*\if@single[3]{%
  \setbox0\hbox{${\mathaccent"0362{#1}}^H$}%
  \setbox2\hbox{${\mathaccent"0362{\kern0pt#1}}^H$}%
  \ifdim\ht0=\ht2 #3\else #2\fi
  }
\newcommand*\rel@kern[1]{\kern#1\dimexpr\macc@kerna}
\newcommand*\widebar[1]{\@ifnextchar^{{\wide@bar{#1}{0}}}{\wide@bar{#1}{1}}}
\newcommand*\wide@bar[2]{\if@single{#1}{\wide@bar@{#1}{#2}{1}}{\wide@bar@{#1}{#2}{2}}}
\newcommand*\wide@bar@[3]{%
  \begingroup 
  \def\mathaccent##1##2{ %
    \if#32 \let\macc@nucleus\first@char \fi
   \setbox\z@\hbox{$\macc@style{\macc@nucleus}_{}$} %
    \setbox\tw@\hbox{$\macc@style{\macc@nucleus}{}_{}$}%
    \dimen@\wd\tw@
    \advance\dimen@-\wd\z@
    \divide\dimen@ 3
    \@tempdima\wd\tw@
    \advance\@tempdima-\scriptspace
    \divide\@tempdima 10
    \advance\dimen@-\@tempdima
    \ifdim\dimen@>\z@ \dimen@0pt\fi
    \rel@kern{0.6}\kern-\dimen@
    \if#31
     \overline{\rel@kern{-0.6}\kern\dimen@\macc@nucleus\rel@kern{0.4}\kern\dimen@}%
      \advance\dimen@0.4\dimexpr\macc@kerna
      \let\final@kern#2%
      \ifdim\dimen@<\z@ \let\final@kern1\fi
      \if\final@kern1 \kern-\dimen@\fi
    \else
      \overline{\rel@kern{-0.6}\kern\dimen@#1}%
    \fi
  }%
  \macc@depth\@ne
  \let\math@bgroup\@empty \let\math@egroup\macc@set@skewchar
  \mathsurround\z@ \frozen@everymath{\mathgroup\macc@group\relax}%
  \macc@set@skewchar\relax
  \let\mathaccentV\macc@nested@a
  \if#31
    \macc@nested@a\relax111{#1}%
  \else
    \def\gobble@till@marker##1\endmarker{}%
    \futurelet\first@char\gobble@till@marker#1\endmarker
    \ifcat\noexpand\first@char A\else
      \def\first@char{}%
    \fi
    \macc@nested@a\relax111{\first@char}%
  \fi
  \endgroup
}
\makeatother

\title{On the Existence of a Global Neighbourhood}
\author{Tom Coates}
\email{t.coates@imperial.ac.uk}
\author{Hiroshi Iritani}
\email{iritani@math.kyoto-u.ac.jp}
\begin{document}
\maketitle 

\begin{abstract}
  Suppose that a complex manifold $M$ is locally embedded into 
  a higher-dimensional neighbourhood as a submanifold.  We show that, if
  the local neighbourhood germs are compatible in a suitable sense,
  then they glue together to give a global neighbourhood of $M$.  As an application, we prove a global version of Hertling--Manin's  unfolding theorem for germs of TEP structures; this has applications in the study of quantum cohomology.
\end{abstract}
\section{Introduction}

We prove:
\begin{theorem}
  \label{thm:main_result}
  Let $M$ be a complex manifold of dimension $m$ and let $\cA$ be a
  sheaf of $\C$-algebras over $M$.  Suppose that there exist a natural
  number $n$ and a morphism of sheaves of $\C$-algebras $\pi \colon
  \cA \to \cO_M$ such that for each $x \in M$ there exists an open
  neighbourhood $U$ of $x$ and an isomorphism $\cA|_U \cong \iota^{-1}
  \cO_{U \times \C^n}$, where $\iota \colon U \hookrightarrow U \times
  \C^n$ is the embedding $x \mapsto (x,0)$, such that the following
  diagram commutes:
  \begin{align}
  \label{eq:local-condition}
  \begin{split} 
    \xymatrix{
      \cA|_U \ar[rr]^{\sim} \ar[rd]_{\pi|_U}&&  \iota^{-1} \cO_{U \times
        \C^n} \ar[dl]^{\iota^\#} \\
      & \cO_U
    }
    \end{split} 
  \end{align}
  Here $\iota^\#$ is the canonical morphism induced by $\iota$.  Then:
  \begin{enumerate}
    \renewcommand{\theenumi}{\roman{enumi}}
  \item There exist a complex manifold $M'$ of dimension $m+n$
    and a closed embedding $\iota \colon M \to M'$ such that we have
    an isomorphism $\cA \cong \iota^{-1} \cO_{M'}$ of sheaves of
    $\C$-algebras which commutes with the surjections to $\cO_M$.
  \item The manifold-germ $M'$ is unique up to unique isomorphism in
    the following sense.  If we have two complex manifolds $M_1'$,
    $M_2'$, two closed embeddings $\iota_1 \colon M \to M'_1$,
    $\iota_2 \colon M\to M'_2$, and an isomorphism of sheaves of
    $\C$-algebras $\phi \colon \iota_1^{-1} \cO_{M_1'} \to
    \iota_2^{-1} \cO_{M_2'}$ commuting with the surjections to
    $\cO_M$, then $\phi$ is induced by a biholomorphic map $\varphi
    \colon N_2 \to N_1$ between open neighbourhoods $N_i$ of $M$ in
    $M_i'$, $i\in\{1,2\}$, such that $\varphi$ is the identity map on
    $M$.  Such a map $\varphi$ is unique as a germ of maps on a
    neighbourhood of $M$.
  \end{enumerate}
\end{theorem}

We also discuss the extension of sheaves, proving:

\begin{theorem}
  \label{thm:with_modules} 
  Suppose that $\iota\colon M\hookrightarrow M'$ is a closed embedding
  of complex manifolds.  Let $\cA$ be the sheaf of $\C$-algebras
  $\cA = \iota^{-1} \cO_{M'}$ over $M$ and let $\cB$ be a coherent $\cA$-module. 
  Then:
  \begin{enumerate}
    \renewcommand{\theenumi}{\roman{enumi}}
  \item There exist an open neighbourhood $N$ of $M$ in $M'$ and a
    coherent $\cO_{N}$-module $\cB'$ on $N$ such that $\iota^{-1} \cB'
    \cong \cB$ as sheaves of $\cA$-modules. 
  \item The sheaf-germ $\cB'$ is unique up to unique isomorphism in the
    following sense.  If we have two coherent $\cO_N$-modules $\cB_1'$
    and $\cB_2'$ on a neighbourhood $N$ of $M$ in $M'$ and
    isomorphisms $\iota^{-1} \cB_1' \cong \cB \cong \iota^{-1} \cB_2'$
    of $\cA$-modules, with $\phi \colon \iota^{-1} \cB_1' \to
    \iota^{-1} \cB_2'$ denoting the composite isomorphism, then $\phi$
    is induced by an isomorphism $\Phi \colon \cB_1'|_P \to
    \cB_2'|_P$ of $\cO_P$-modules on an open neighbourhood $P$ of $M$
    in $N$.  Such a morphism $\Phi$ is unique as a germ of
    homomorphisms over a neighbourhood of $M$.
  \end{enumerate}
\end{theorem}

\begin{remark} 
By Oka's coherence theorem, $\cO_{M'}$ is coherent 
and hence $\cA = \iota^{-1} \cO_{M'}$ is also coherent as a sheaf of algebras. 
Therefore being a coherent $\cA$-module is equivalent to being 
\emph{locally finitely presented} as an $\cA$-module, 
i.e.~for each $x\in M$, there exists an open neighbourhood 
$U$ of $x$ in $M$ and an exact sequence of $\cA$-modules:
  \begin{equation} 
    \label{eq:presentation}
    \xymatrix{
      \cA_U^{\oplus k} \ar[r] &  
      \cA_U^{\oplus l} \ar[r] &
      \cB|_U \ar[r] &  0
      }
  \end{equation} 
  for some $k, l\in \N$. 
See e.g.~\cite[Appendix]{Kashiwara}. 
\end{remark}

\begin{remark} 
  If in addition $\cB$ is locally free as an $\cA$-module 
in Theorem \ref{thm:with_modules}, then 
$\cB'$ becomes locally free as an $\cO_N$-module in a
  neighbourhood $N$ of $M$, because the stalk $\cB'_x$ at each $x\in M$
  is a free $\cO_{M',x}$-module.
\end{remark} 

\begin{remark}
\label{rem:realanalytic} 
We have stated Theorems~\ref{thm:main_result}
and~\ref{thm:with_modules} in the category of holomorphic manifolds,
but the same statements hold true, with the same proofs, in the real
analytic category.  
\end{remark} 

We can reformulate our results as an \emph{equivalence of categories}.
Namely, for Theorem \ref{thm:main_result}, the category of sheaves
$\cA$ of $\C$-algebras on $M$ equipped with surjections $\pi \colon
\cA \to \cO_M$ satisfying the local condition
\eqref{eq:local-condition} is equivalent to the category of germs of
neighbourhoods $\iota \colon M \hookrightarrow M'$ of $M$.  For
Theorem \ref{thm:with_modules}, the category of coherent 
$\cA$-modules is equivalent to the category 
of germs of coherent sheaves on a neighbourhood of $M$ in $M'$.  It is
not difficult to modify the discussion below to prove these
categorical equivalences.

In the real analytic category (Remark \ref{rem:realanalytic}), 
Theorem \ref{thm:main_result} may be viewed as a generalization of 
the existence theorem for the complexification of a real-analytic manifold, 
see e.g.~\cite{Whitney--Bruhat}.  
In the $C^\infty$ category, Lemma
\ref{lem:basicgluing} below is not valid (see e.g.~\cite{Shiota}) and our results
do not hold.  However, most of the arguments for Theorem
\ref{thm:main_result} work if we can take representatives of
neighbourhood germs and $C^\infty$ gluing maps between them which
satisfy cocycle conditions as germs; similar arguments appear in the
context of Kuranishi structures \cite{FOOO,Joyce}.
Our original motivation was to globalize the unfolding of
Frobenius-type structures (or meromorphic connections, or TEP
structures) which has been studied by Hertling--Manin\cite{Hertling--Manin} and Reichelt~\cite{Reichelt} on the level of
germs.  We give a global version of Hertling--Manin's unfolding theorem for TEP structures in section~\ref{sec:unfolding} below.

\begin{notation}
  We require that manifolds be paracompact.
\end{notation}

\begin{notation}
  We write $A \Subset B$ if and only if $A$ is a relatively compact
  subset of $B$.
\end{notation}

\begin{acknowledgments*} 
  T.C.~thanks Eugenia Cheng for useful correspondence.  H.I.~thanks Ono
  Kaoru and Ken-ichi Yoshikawa for useful discussions.
\end{acknowledgments*}

\section{The Proof of Theorem~\ref{thm:main_result}}
Our assumptions on $\cA$ imply that we can find a locally-finite open
covering $\{W_i: i\in I\}$ of $M$ with index set $I$ such that
$\cA|_{W_i} \cong \iota^{-1} \cO_{W_i \times Z_i}$, where $Z_i$ is a
copy of $\C^n$ and $\iota \colon W_i \to W_i \times Z_i $ is the
embedding $x \mapsto (x,0)$.  Without loss of generality we can assume
both that $W_i$ is a co-ordinate neighbourhood on $M$ (i.e.~is
identified with an open subset of $\C^m$) and that $W_i$ is relatively
compact in $M$.  We take locally finite coverings $\{U_i: i\in I\}$,
$\{V_i : i\in I\}$ with the same index set $I$ such that $U_i \Subset
V_i \Subset W_i$.  We write $W_{ij} := W_i \cap W_j$, $V_{ij} := V_i
\cap V_j$, $V_{ijk} := V_i \cap V_j \cap V_k$.  The basic fact we use
for the gluing is the following Lemma.

\begin{lemma}
\label{lem:basicgluing} 
Let $U \subset \C^m$ be an open set and let $\iota \colon U \to U
\times \C^n$ be the embedding $x\mapsto (x,0)$.  Let $\phi \colon
\iota^{-1}\cO_{U \times \C^n} \to \iota^{-1} \cO_{U\times \C^n}$ be a
homomorphism of sheaves of $\C$-algebras which commutes with the
natural surjections to $\cO_U$.  Then:
\begin{itemize}
\item[(a)] there exists an open neighbourhood $U'$ of $U\times \{0\}$
  in $U\times \C^n$ and a holomorphic map $\varphi \colon U' \to
  U\times \C^n$ which is the identity on $U\times \{0\}$ such that
  $\phi$ coincides with the pull-back by $\varphi$; 
\item[(b)] if $\phi$ is an isomorphism then the map $\varphi$ is a
  biholomorphic isomorphism onto its image.
\end{itemize}
\end{lemma} 
\begin{proof}[Proof of Lemma \ref{lem:basicgluing}] 
  Statement (a) implies statement (b), by the inverse function theorem,
  so we prove (a).  Consider first the case where $m=0$ and $U$ is a
  point. Then $\phi$ is a \mbox{$\C$-algebra} endomorphism of the ring
  $\iota^{-1} \cO_{\C^n} = \C\{z_1,\dots,z_n\}$ of convergent power
  series which preserves the maximal ideal $\fm = (z_1,\dots,z_n)$.
  Because such a $\phi$ is continuous with respect to the $\fm$-adic
  topology, $\phi$ is determined by the images of the generators
  $z_1,\dots,z_n$.  The images determine a holomorphic map $\varphi
  \colon (z_1,\dots,z_n) \mapsto (\phi(z_1),\dots, \phi(z_n))$ which
  is defined on a neighbourhood $U'$ of $0$ in $\C^n$, and $\phi$
  coincides with the pull-back by $\varphi$.

  Consider now the general case.  Let $t_1,\dots,t_m$ denote the
  standard co-ordinates on $U \subset \C^m$ and let $z_1,\dots,z_n$
  denote the standard co-ordinates on $\C^n$. Then the images of
  $t_1,\dots,t_m$ and $z_1,\dots,z_n$ under $\phi$ give global
  sections of $\iota^{-1}\cO_{U\times \C^m}$, and thus they define a
  holomorphic map
  \[
  \varphi \colon (t_1, \dots, t_m, z_1,\dots,z_n) \mapsto 
  (\phi(t_1),\dots,\phi(t_m),\phi(z_1),\dots,\phi(z_n))
  \]
  on a neighbourhood $U'$ of $U\times \{0\}$ in $U\times \C^n$.  Since
  $\phi$ commutes with the surjections to $\cO_U$, $\varphi$ restricts
  to the identity map on $U\times \{0\}$.  The pull-back
  $\varphi^\star \colon \iota^{-1}\cO_{U\times \C^n} \to \iota^{-1}
  \cO_{U\times \C^n}$ defines a homomorphism of sheaves of
  \mbox{$\C$-algebras} commuting with the surjections to $\cO_U$.  The
  $m=0$, $U = \text{point}$ case implies that $(\varphi^\star)_t =
  \phi_t$ on the stalk at every point $t\in U$.  Thus $\varphi^\star=
  \phi$.
\end{proof} 

The composite isomorphism $\iota^{-1} \cO_{W_{ij}\times Z_i} \cong
\cA|_{W_{ij}} \cong \iota^{-1} \cO_{W_{ij}\times Z_j}$ induces a
biholomorphic isomorphism $\varphi_{ij} \colon N_{ij} \to N_{ji}$ for
each $i,j\in I$, where $N_{ij}$ is an open neighbourhood of
$W_{ij}\times \{0\}$ in $W_{ij} \times Z_i$ and $\varphi_{ij}$ is the
identity on $W_{ij} \times \{0\}$.  Note that $N_{ij}$ and $N_{ji}$
are subsets of different spaces.  Without loss of generality we may
assume that $N_{ii} = W_i \times Z_i$ and that $\varphi_{ii}$ is the
identity map.  Define:
\begin{align*}
  & O_i := V_i \times Z_i \\
  & O_{ij} := (V_{ij} \times Z_i) \cap N_{ij} \cap \varphi_{ij}^{-1}
  (V_{ij} \times Z_j)
\end{align*}
Then $O_{ij}$ is an open subset of $O_i$ which contains $V_{ij} \times
\{0\}$.  By restricting $\varphi_{ij}$ to $O_{ij}$, we obtain a
biholomorphic isomorphism $\varphi_{ij} \colon O_{ij} \to O_{ji}$ such
that $\varphi_{ij}|_{V_{ij} \times \{0\}}$ is the identity map.

\begin{lemma}
  \label{lem:Q_i}
  There exist open subsets $Q_i$ of $O_i$, $i \in I$, such that for
  each $i$,~$j \in I$ we have:
  \begin{itemize}
  \item[(a)] $Q_i \Subset O_i$;
  \item[(b)] $U_i \times \{0\} \subset Q_i \subset U_i \times Z_i$;
  \item[(c)] $Q_{ij} \Subset O_{ij}$, where $Q_{ij} := Q_i \cap O_{ij} \cap
    \varphi_{ij}^{-1}(Q_j)$.
  \end{itemize}
\end{lemma}

\begin{proof}[Proof of Lemma~\ref{lem:Q_i}]
  Denote by $\barU_i$ the closure of $U_i$ in $V_i$; by assumption
  $\barU_i$ is compact.  We have that $\barU_i \cap \barU_j$ is
  contained in $V_{ij}$, and hence that $(\barU_i \cap \barU_j) \times
  \{0\} \subset O_{ij}$.  Fix $i$,~$j \in I$, and fix a relatively
  compact open subset $P$ of $O_{ij}$ such that $P$ contains $(\barU_i
  \cap \barU_j) \times \{0\}$; such a subset exists because the set
  $\barU_i \cap \barU_j$ is compact.  Set:
  \[
  Q_i(n) := U_i \times \big\{x \in Z_i : |x| < \textstyle \frac{1}{n} \big\}
  \]
  noting that $Q_i(n)$ satisfies conditions (a) and (b) of the Lemma.

  We claim that there exists $n$ such that
  \begin{equation}
    \label{eq:Qin_goal}
    Q_i(n) \cap O_{ij} \cap \varphi_{ij}^{-1} \big(Q_j(n)\big) \subset P
  \end{equation}
  Suppose, on the contrary, that for each $n$ there exists an element
  $x_n \in Q_i(n) \cap O_{ij} \cap \varphi_{ij}^{-1} \big(Q_j(n)\big)$
  such that $x_n \not \in P$.  After passing to a subsequence, we have
  that $(x_n)$ converges to a limit $x \in \barU_i \times \{0\}$.
  Thus $(x_n)$ converges in $O_i$.  On the other hand each $x_n$ lies
  in the closed subset $O_i \setminus P$ of $O_i$, and so $x \in O_i
  \setminus P$.  Thus $x$ lies in $(\barU_i \setminus \barU_j) \times
  \{0\}$.  Now $x_n$ lies in $O_{ij}$ for each $n$, hence $x_n$ lies
  in $V_{ij} \times Z_i$ and the limit $x$ lies in the closure of
  $V_{ij}$ in $M$.  However the closure of $V_{ij}$ is contained in
  $W_{ij}$.  Recall that $\varphi_{ij}$ is defined and continuous on
  the open neighbourhood $N_{ij}$ of $W_{ij} \times \{0\}$ in $W_i'$.
  Thus:
  \[
  \varphi_{ij}(x) = \lim_{n \to \infty} \varphi_{ij}(x_n)
  \]
  The right hand side here converges to an element in $\barU_j \times
  \{0\}$, since $\varphi_{ij}(x_n) \in Q_j(n)$.  This is a
  contradiction: we have shown that $x \in \barU_i \setminus \barU_j$,
  and $\varphi_{ij}|_{W_{ij} \times \{0\}}$ is the identity map.
  Thus, for each $i$ and $j \in I$, there exists an integer $n = n(i,j)$ such that
  \eqref{eq:Qin_goal} holds.

  Since $V_i$ is relatively compact in $M$ and since the covering
  $\{V_i : i \in I\}$ is locally finite, only finitely many $V_j$ have
  non-empty intersection with a fixed $V_i$.  Thus we can define:
  \begin{align*}
    & n(i) := \max \{ n(i,j) : \text{$j \in I$ such that $V_i \cap V_j \ne \varnothing$}
    \} \\
    & Q_i := Q_i\big(n(i)\big)
  \end{align*}
  to obtain open sets $\{Q_i : i \in I\}$ with the properties claimed.
\end{proof}

\begin{lemma}
  \label{lem:closed_image}
  Let $\{Q_i : i \in I\}$ be such that $Q_i$ is an open subset of
  $O_i$ and that properties (a--c) in Lemma~\ref{lem:Q_i} hold.  Then
  the image of the map $Q_{ij} \to Q_i \times Q_j$ given by
  $(\text{\rm inclusion}, \varphi_{ij})$ is closed in $Q_i \times
  Q_j$.
\end{lemma}

\begin{proof}[Proof of Lemma~\ref{lem:closed_image}] Let $(x_n)$ be a
  sequence in $Q_{ij}$ such that $(x_n)$ converges in $Q_i$ and
  $\big(\varphi_{ij}(x_n)\big)$ converges in $Q_j$.  Let $x$ denote
  the limit of $(x_n)$ in $Q_i$.  Since $Q_{ij}$ is relatively compact
  in $O_{ij}$, the limit $x$ lies in $O_{ij}$.  But
  $\big(\varphi_{ij}(x_n)\big)$ converges in $Q_j$ and thus
  $\varphi_{ij}(x) \in Q_j$, or in other words $x \in
  \varphi_{ij}^{-1}(Q_j)$.  Thus $x \in Q_{ij}$.
\end{proof}

\begin{lemma}
  \label{lem:Q_i_extra}
  There exist open subsets $Q_i$ of $O_i$, $i \in I$, such that
  properties (a--c) in Lemma~\ref{lem:Q_i} hold and further, for each
  $i$,~$j$,~$k \in I$, we have:
  \begin{itemize}
  \item[(d)] $Q_{ij} \cap Q_{ik} \subset \varphi_{ij}^{-1}(O_{jk})$;
  \item[(e)] $\varphi_{jk} \circ \varphi_{ij} = \varphi_{ik}$ on
    $Q_{ij} \cap Q_{ik}$.
  \end{itemize}
Note that condition (d) guarantees that the composition in (e) is well-defined.
\end{lemma}

\begin{proof}[Proof of Lemma~\ref{lem:Q_i_extra}]
  Let $\{Q_i : i \in I\}$ be such that, for each $i \in I$, $Q_i$ is
  an open subset of $O_i$ and that properties (a--c) hold.  Such
  subsets exist by Lemma~\ref{lem:Q_i}.  Define:
  \[
  Q_{ijk} := Q_{ij} \cap Q_{ik}
  \]
  Let $\barQ_{ij}$ denote the closure of $Q_{ij}$ in $O_{ij}$. This is
  compact, and hence $\barQ_{ij}$ is at the same time the closure of
  $Q_{ij}$ in $O_i$.  Let $\barQ_{ijk}$ denote the closure of
  $Q_{ijk}$ in $O_i$.  This is contained in the compact set
  $\barQ_{ij} \cap \barQ_{ik}$, hence in particular is contained in
  $O_{ij} \cap O_{ik}$.

We claim that there exists an open neighbourhood 
$N_{ijk}$ of $\barQ_{ijk}
\cap (U_i \times \{0\})$ in $O_{ij} \cap O_{ik}$ such that:
\begin{itemize}
  \item $N_{ijk} \subset \varphi_{ij}^{-1} (O_{jk})$; and
  \item $\varphi_{jk} \circ \varphi_{ij} = \varphi_{ik}$ holds on $N_{ijk}$.
  \end{itemize}
  It suffices to show that each $x$ in $\barQ_{ijk} \cap (U_i \times
  \{0\})$ has an open neighbourhood $N_x$ in $O_{ij} \cap O_{ik}$ such
  that $N_x \subset \varphi_{ij}^{-1} (O_{jk})$ and that $\varphi_{jk}
  \circ \varphi_{ij} = \varphi_{ik}$ holds on $N_x$.  Let $x \in
  \barQ_{ijk} \cap (U_i \times \{0\})$.  Then $x$ lies in $O_{ij} \cap
  O_{ik} \cap (U_i\times \{0\}) = (V_{ijk} \cap U_i) \times \{0\}$ and
  hence lies in $O_{ij} \cap O_{ik} \cap \varphi_{ij}^{-1}(O_{jk})$.
  But $\varphi_{jk} \circ \varphi_{ij} = \varphi_{ik}$ holds as germs
  at $x$, and so choosing $N_x$ to be a sufficiently small open
  neighbourhood of $x$ in $O_{ij} \cap O_{ik} \cap
  \varphi_{ij}^{-1}(O_{jk})$ proves the claim.

  Recall that properties (a--c) in Lemma~\ref{lem:Q_i} hold for $\{Q_i
  : i \in I\}$ and note that, with the exception of the assertion that
  $U_i \times \{0\} \subset Q_i$, these properties are preserved under
  shrinking the sets $Q_i$.  For a fixed $i\in I$, there are only
  finitely many pairs $(j,k)\in I\times I$ such that the triple
  intersection $V_{ijk}$ is nonempty.  Thus, as $Q_{ijk} \subset
  V_{ijk}\times Z_i$, there are only finitely many pairs $(j,k)\in
  I\times I$ such that $Q_{ijk}$ is nonempty.  Let
  $(j_1,k_1),\dots,(j_f,k_f)$ be all such pairs.  We shrink $Q_i$
  inductively as follows: set $Q_i^{(0)} := Q_i$, set
  \[ 
  Q_i^{(a)} := (Q_i^{(a-1)} \setminus \barQ_{ij_ak_a}) 
  \cup (N_{ij_ak_a} \cap Q_i^{(a-1)}) 
  \subset Q_i^{(a-1)} 
  \] 
  for $a=1,\dots,f$, and set $Q_i^\new := Q_i^{(f)}$.  Then $Q_i^\new$
  is an open subset of the original set $Q_i$, and it contains $U_i
  \times \{0\}$.  Thus properties (a--c) in Lemma~\ref{lem:Q_i} hold
  for the new sets $\{Q_i^\new : i \in I\}$.  Furthermore, for each
  $i$,~$j$,~$k \in I$, $Q^\new_{ij} := Q_i^\new \cap O_{ij} \cap
  \varphi_{ij}^{-1}(Q_j^\new)$ is contained in $Q_{ij}$ and
  $Q_{ijk}^\new := Q_{ij}^\new \cap Q_{ik}^\new$ is contained in
  $Q_{ijk}$.  Also $Q_{ijk}^\new$ is contained in $N_{ijk}$.  
  Therefore properties (d) and (e) hold for $\{Q_i^\new: i \in I\}$,
  and the Lemma is proved.
\end{proof}

\begin{remark*}
  Let $Q_i \subset O_i$, $i \in I$, be open subsets such that
  properties (a--e) in Lemma \ref{lem:Q_i_extra} hold.  
In particular,
  then, $Q_{ij} \cap Q_{ik} \subset \varphi_{ij}^{-1}(O_{jk})$.  But
  slightly more is true: in fact $Q_{ij} \cap Q_{ik} \subset
  \varphi_{ij}^{-1}(Q_{jk})$.  For if $x \in Q_{ij} \cap Q_{ik}$ then
  $\varphi_{ij}(x) \in Q_{ji} \subset Q_j$ and $\varphi_{ik}(x) \in
  Q_{ki} \subset Q_k$.  Also $\varphi_{jk} \circ \varphi_{ij}(x) =
  \varphi_{ik}(x)$, which lies in $Q_k$.  Thus $\varphi_{ij}(x)$ lies
  in $Q_j \cap O_{jk} \cap \varphi_{jk}^{-1}( Q_k ) =: Q_{jk}$.
\end{remark*}

We now complete the proof of Theorem~\ref{thm:main_result}.  Choose
open subsets $Q_i \subset O_i$, $i \in I$, such that properties (a--e)
in Lemma \ref{lem:Q_i_extra} hold.  (This is possible by
Lemma~\ref{lem:Q_i_extra}.) 
Set $M'$ equal to the quotient space:
\[
\left(\coprod_{i \in I} Q_i \right) \!\!\Bigg/ \!\! \sim
\]
by the equivalence relation $\sim$ generated by $x \sim
\varphi_{ij}(x)$ where $x \in Q_{ij}$.  We claim that $M'$ is a
complex manifold.

Let $X = \coprod_{i \in I} Q_i$ and consider the binary relation $R$
on $X \times X$ given by:
\[
R := \coprod_{i,j \in I} Q_{ij} \subset \coprod_{i,j \in I} Q_i \times Q_j
= X \times X
\]
where the map $Q_{ij} \to Q_i \times Q_j$ is given by $(\text{\rm
  inclusion}, \varphi_{ij})$.  Then $M'$ is the quotient space of $X$
by the equivalence relation generated by $R$.  To show that $M'$ is a
complex manifold it suffices to prove that $M'$ is Hausdorff; hence it
suffices to prove that $R$ is closed and that $R$ is an equivalence
relation.  We have shown that the image of the map $Q_{ij} \to Q_i
\times Q_j$ is closed (Lemma~\ref{lem:closed_image}), so $R$ is
closed.  It remains to show that $R$ is an equivalence relation.
Reflexivity ($x \sim x$) is obvious, since $Q_{ii}= Q_i$ and
$\varphi_{ii}$ is the identity map.  Symmetricity ($x \sim y \implies
y \sim x$) is also obvious, since $\varphi_{ij}$ and $\varphi_{ji}$
are inverse to each other.  For transitivity ($x \sim y \wedge y \sim
z \implies x \sim z$) assume that $x \in Q_j$, $y\in Q_i$, $z \in
Q_k$, $x \sim y$, and $y \sim z$.  Then $y \in Q_{ij}$ (since $y \sim
x$) and $y \in Q_{ik}$ (since $y \sim z$), thus $y \in Q_{ij} \cap
Q_{ik}$.  The Remark after Lemma~\ref{lem:Q_i_extra} implies that $y
\in \varphi_{ij}^{-1}(Q_{jk})$, and Lemma~\ref{lem:Q_i_extra} implies
that $\varphi_{jk} \circ \varphi_{ij}(y) = \varphi_{ik}(y)$.  But $x =
\varphi_{ij}(y)$ and $z = \varphi_{ik}(y)$, so $z = \varphi_{jk}(x)$.
Thus $x \sim z$.  It follows that $R$ is an equivalence relation, and
that $M'$ is a complex manifold.  It is clear that $M$ is a closed
submanifold of $M'$.  This completes the proof of part (i) of
Theorem~\ref{thm:main_result}.

Let us prove part (ii) of Theorem~\ref{thm:main_result}.  Suppose we
have two closed embeddings $\iota_1 \colon M \to M_1'$, $\iota_2
\colon M \to M_2'$ and an isomorphism $\phi \colon \iota_1^{-1}
\cO_{M_1'} \cong \iota_2^{-1} \cO_{M_2'}$ of sheaves of $\C$-algebras
commuting with natural surjections to $\cO_M$.  By Lemma
\ref{lem:basicgluing}, the isomorphism $\phi \colon \iota_1^{-1}
\cO_{M_1'} \to \iota_2^{-1} \cO_{M_2'}$ is locally induced by a
biholomorphic map which is the identity on $M$.  Therefore we have a
locally finite open covering $\{S_i : i\in I\}$ of $M$, open
neighbourhoods $T_i$ of $S_i$ in $M'_2$, and holomorphic maps $\varphi
\colon T_i \to M_1'$ such that $\varphi_i$ is the identity map on $T_i
\cap M$ and $\phi|_{S_i}= \varphi_i^\star$.  Without loss of
generality we may assume that $S_i$ is relatively compact in $M$.
Choose an open covering $\{R_i: i\in I\}$ of $M$ such that $R_i
\Subset S_i$, and choose an open tubular neighbourhood of $M$ in
$M_2'$.  The tubular neighbourhood here is identified with a
neighbourhood of the zero section of the normal bundle of $M$ in
$M'_2$; we choose a (fibrewise) Riemannian metric on it.  We write
$U(\epsilon)\subset M_2'$ for the open tube of length $\epsilon>0$
over an open subset $U \subset M$.  The maps $\varphi_i$ and
$\varphi_j$ on the overlap $T_{ij} := T_i \cap T_j$ coincide on an
open neighbourhood $T_{ij}^{\circ}\subset T_{ij}$ of $S_{ij} := S_i
\cap S_j$.  Since, for fixed $i \in I$, there are only finitely many
$j\in I$ such that $S_{ij}$ is nonempty, there exists $\epsilon_i >0$
such that:
\begin{itemize} 
\item $R_i(\epsilon_i) \subset T_i$ 
\item $R_{ij}(\epsilon_i) \subset T_{ij}^{\circ}$ for all $j \in I$,
  where $R_{ij} := R_i \cap R_j \Subset S_{ij}$.
\end{itemize} 
Then the maps $\{\varphi_i|_{R_i(\epsilon_i)}: i\in I\}$ coincide over
each overlap $R_i(\epsilon_i) \cap R_j( \epsilon_j) \subset
T_{ij}^{\circ}$, $i$,~$j \in I$, and thus determine a global
holomorphic map $\varphi$ on $N = \bigcup_{i\in I} R_i(\epsilon_i)
\subset M_2'$.  The uniqueness of $\varphi$ as a germ is obvious.
This completes the proof of Theorem~\ref{thm:main_result}.

\section{The Proof of Theorem \ref{thm:with_modules}}

\begin{lemma} 
  \label{lem:coverings} 
  Let $\iota \colon M \to M'$ be a closed embedding of complex
  manifolds, let $\cA$ be the sheaf of $\C$-algebras $\cA = \iota^{-1}
  \cO_{M'}$ over $M$, and let $\cB$ be a locally finitely presented
  $\cA = \iota^{-1} \cO_{M'}$-module.  There exist:
  \begin{itemize} 
  \item an open covering $\{V_i : i \in I\}$ of $M$ such that $V_i$ is
    relatively compact in $M$;
  \item for each $i \in I$, an open subset $W'_i$ of $M'$ such that
    $V_i\subset W_i'$;
  \item for each $i$,~$j \in I$, an open subset $A_{ij}$ of $M'$ such
    that $V_{ij} \subset A_{ij} \subset W'_{ij}$, where $V_{ij} := V_i
    \cap V_j$ and $W'_{ij} := W'_i \cap W'_j$;
  \item for each $i$,~$j$,~$k\in I$, an open subset $B_{ijk}$ of $M'$
    such that $V_{ijk} \subset B_{ijk} \subset A_{ijk}$, where
    $V_{ijk} := V_i \cap V_j \cap V_k$ and $A_{ijk} := A_{ij} \cap
    A_{jk} \cap A_{ik}$;
  \item for each $i \in I$, a coherent $\cO_{W_i'}$-module $\cB'_i$ on
    $W'_i$;
  \item for each $i \in I$, an isomorphism $\theta_i \colon \iota^{-1}
    \cB'_i \cong \cB|_{V_i}$ of $\cA_{V_i}$-modules;
  \item for each $i$,~$j \in I$, an isomorphism 
    $\phi_{ij} \colon \cB'_i|_{A_{ij}} \cong \cB'_j|_{A_{ij}}$ 
    of $\cO_{A_{ij}}$-modules;
  \end{itemize} 
  such that $A_{ij}$, $B_{ijk}$ are symmetric in their indices and
  that the diagrams:
  \begin{align} 
    \label{eq:cocycle} 
    \begin{split} 
      \xymatrix{ 
        \iota^{-1} \cB'_i|_{A_{ij}} \ar[rr]^{\iota^{-1}\phi_{ij}} 
        \ar[rd]_{\theta_i} && \iota^{-1} \cB'_j|_{A_{ij}} 
        \ar[dl]^{\theta_j} \\ 
        & \cB|_{V_{ij}} 
      } 
      \quad 
      \xymatrix{
        \cB'_i|_{B_{ijk}} \ar[rr]^{\phi_{ik}} \ar[rd]_{\phi_{ij}} & & 
        \cB'_k|_{B_{ijk}} \\ 
        & \cB'_j|_{B_{ijk}} \ar[ur]_{\phi_{jk}} 
      }
    \end{split} 
  \end{align} 
  commute for each $i,j,k\in I$. 
\end{lemma}
 
\begin{proof}[Proof of Lemma \ref{lem:coverings}] 
  We take open coverings $\{V_i : i\in I\}$, $\{W_i:i\in I\}$ of $M$
  by Stein open subsets $V_i$, $W_i$ such that $V_i \Subset W_i
  \Subset M$ and that $\cB|_{W_i}$ has a finite presentation
  \[
  \xymatrix{
    \cA_{W_i}^{\oplus k_i} \ar[r]^{\xi_i} &
    \cA_{W_i}^{\oplus l_i} \ar[r]^{\eta_i} &
    \cB|_{W_i} \ar[r] & 0
    }
    \]
    as in \eqref{eq:presentation}.  For example, we can take
    $V_i$,~$W_i$ to be small open balls centered at the same point of a
    co-ordinate chart.  The $\cA$-module homomorphism $\xi_i \colon
    \cA_{W_i}^{\oplus k_i} \to \cA_{W_i}^{\oplus l_i}$ extends to an
    $\cO$-module homomorphism $\xi_i' \colon \cO_{W_i'}^{\oplus k_i}
    \to \cO_{W_i'}^{\oplus l_i}$ on a neighbourhood $W_i'$ of $W_i$ in
    $M'$ and defines a coherent $\cO_{W_i'}$-module $\cB_i'$ by the
    exact sequence:
    \[
    \xymatrix{
      \cO_{W'_i}^{\oplus k_i} \ar[r]^{\xi_i'} &
      \cO_{W'_i}^{\oplus l_i} \ar[r]^{\eta'_i} &
      \cB_i' \ar[r] & 0
    }
    \]
    By construction there is an $\cA$-module isomorphism $\theta_i
    \colon \iota^{-1}\cB_i' \to \cB|_{W_i}$.  For each pair $(i,j)$
    such that $V_{ij} := V_i \cap V_j$ is nonempty, we construct a
    homomorphism $\phi_{ij}$ from $\cB_i'$ to $\cB_j'$.  Let
    $e_1,\dots,e_{l_i}$ denote the standard basis of
    $\cO_{W_i}^{\oplus l_i}$.  For each $1\le a \le l_i$, the image
    $\eta_i(e_a)$ is a section of $\cB'_i$ and induces a section $s_a$
    of $\iota^{-1} \cB'_i \cong \cB|_{W_i}$.  Via the isomorphism
    $\iota^{-1} \cB_j' |_{W_{ij}} \cong \cB|_{W_{ij}}$, the
    restriction $s_a|_{W_{ij}}$ can be lifted to a section $t_a$ of
    $\cB'_j$ over an open neighbourhood $C_{ij}$ of $W_{ij}$ in
    $W_{ij}'$.  Because the intersection $V_{ij}$ of Stein open sets
    $V_i$, $V_j$ is Stein and because $V_{ij} \Subset W_{ij}$, we can
    find a Stein open neighbourhood $A_{ij}$ of $V_{ij}$ in $C_{ij}$.
    Because $A_{ij}$ is Stein, we can find a lift $u_a \in
    \Gamma(A_{ij}, \cO^{\oplus l_j})$ of $t_a|_{A_{ij}}$ such that
    $\eta_j (u_a) = t_a|_{A_{ij}}$.  The sections $u_1,\dots, u_{l_i}$
    define a homomorphism $\psi_{ij} \colon \cO_{A_{ij}}^{\oplus l_i}
    \to \cO_{A_{ij}}^{\oplus l_j}$ sending $e_a$ to $u_a$.
    \[
    \xymatrix{
      \cO_{A_{ij}}^{\oplus k_i} \ar[r]^{\xi_i} &
      \cO_{A_{ij}}^{\oplus l_i} \ar[r]^{\eta_i} \ar[d]_{\psi_{ij}} &
      \cB_i'|_{A_{ij}} \ar[r] & 0 \\ 
      \cO_{A_{ij}}^{\oplus k_j} \ar[r]^{\xi_j} &
      \cO_{A_{ij}}^{\oplus l_j} \ar[r]^{\eta_j} &
      \cB_j'|_{A_{ij}} \ar[r] & 0  
    }
    \]
    We claim that, after shrinking $A_{ij}$ if necessary, $\psi_{ij}$
    induces a homomorphism $\phi_{ij} \colon \cB_i'|_{A_{ij}} \to
    \cB_j'|_{A_{ij}}$.  It suffices to show that the composition
    $\eta_j \circ \psi_{ij} \circ \xi_i$ is zero in a neighbourhood of
    $V_{ij}$.  
By construction, $\eta_i(v)$ and $\eta_j \circ \psi_{ij}(v)$ 
define the same section of $\cB$ over $V_{ij}$ 
for every $v\in \Gamma(A_{ij},\cO_{A_{ij}}^{\oplus l_i})$. 
    Hence for $w\in \Gamma(A_{ij}, \cO_{A_{ij}}^{\oplus k_i})$, 
    $\eta_j \circ \psi_{ij} \circ \xi_i(w)$ and $\eta_i \circ
    \xi_i(w) =0$ define the same section of $\cB$ over $V_{ij}$.  This
    means that $\eta_j \circ \psi_{ij} \circ \xi_i(w)$ vanishes in a
    neighbourhood of $V_{ij}$, and the claim follows.  It is clear
    that the first diagram in \eqref{eq:cocycle} commutes.  We can
    also assume that $A_{ij} = A_{ji}$ by replacing $A_{ij}$ with
    $A_{ij} \cap A_{ji}$ and restricting $\phi_{ij}$ to it if
    necessary.

    Finally we find an open subset $B_{ijk}$ of $A_{ijk} = A_{ij} \cap
    A_{jk} \cap A_{ij}$ containing $V_{ijk} = V_i \cap V_j \cap V_k$
    on which the second diagram in \eqref{eq:cocycle} commutes
    (i.e.~the cocycle condition holds). But this is straightforward,
    because $\phi_{jk} \circ \phi_{ij} = \phi_{ik}$ holds at the stalk
    of each $x\in V_{ijk}$.
\end{proof} 

Let $\iota\colon M\hookrightarrow M'$, $\cA$, and $\cB$ be as in the
statement of Theorem~\ref{thm:with_modules}.  Take the data
constructed in Lemma \ref{lem:coverings}.  By taking a refinement if
necessary, we may assume that the open covering $\{V_i : i\in I\}$ is
locally finite.  Choose an open covering $\{U_i : i\in I\}$ of $M$
such that $U_i \Subset V_i$.  Take an open tubular neighbourhood of
$M$ in $M'$ and, as in the proof of Theorem~\ref{thm:main_result}, fix
a fibrewise Riemannian metric on it.  For a open subset $U$ of $M$ and
$\epsilon>0$, we denote by $U(\epsilon)\subset M'$ the open tube of
length $\epsilon$ over $U$. For each $i\in I$, there are only finitely
many $j\in I$ such that the intersection $V_{ij}$ is non-empty.
Therefore we can find $\epsilon_i >0$ such that:
\begin{itemize} 
\item $U_i(\epsilon_i) \subset W_i'$;  
\item $U_{ij}(\epsilon_i) \subset A_{ij}$ for all $j \in I$, where
  $U_{ij}:= U_i \cap U_j$;
\item $U_{ijk}(\epsilon_i) \subset B_{ijk}$ for all $j$,~$k \in I$,
  where $U_{ijk} := U_i \cap U_j \cap U_k$.
\end{itemize} 
Then the coherent sheaves $\cB_i'|_{U_i(\epsilon_i)}$ glue together
via the homomorphisms $\phi_{ij}|_{U_{i}(\epsilon_i) \cap
  U_j(\epsilon_j)}$ to give a global coherent $\cO_N$-module $\cB'$ on
$N = \bigcup_{i\in I} U_i(\epsilon_i)$.  It is clear that there is an
isomorphism $\iota^{-1}\cB' \cong \cB$ of $\cA$-modules.  This
completes the proof of part (i) of Theorem \ref{thm:with_modules}.

Let us prove part (ii) of Theorem \ref{thm:with_modules}.  Suppose
that we have coherent $\cO_N$-modules $\cB_1'$ and $\cB_2'$ and a
isomorphisms $\iota^{-1}\cB_1' \cong \cB \cong \iota^{-1} \cB_2'$ of
$\cA$-modules.  Let $\phi \colon \iota^{-1}\cB_1' \to \iota^{-1}
\cB_2'$ denote the composite isomorphism. We can find a locally finite
open covering $\{S_i :i\in I\}$ of $M$ together with a family $\{T_i :
i\in I\}$ of open subsets of $M'$ such that $S_i \Subset T_i$ and
$\cB_1'|_{T_i}$ and $\cB_2'|_{T_i}$ have finite presentations as
$\cO_{T_i}$-modules.  Without loss of generality we may assume that
$S_i$ is Stein.  The argument in the proof of Lemma
\ref{lem:coverings} shows that we can find an open neighbourhood
$T_i^\circ$ of $S_i$ in $T_i$ and a homomorphism $\Phi_i \colon
\cB_1'|_{T_i^\circ} \to \cB_2'|_{T_i^{\circ}}$ such that $\iota^{-1}
\Phi_i = \phi|_{S_i}$.  For each $i$,~$j\in I$, there exists an open
subset $P_{ij}$ of $T_i^\circ \cap T_j^\circ$ containing $S_{ij} :=
S_i \cap S_j$ such that $\Phi_i |_{P_{ij}} = \Phi_j|_{P_{ij}}$.  Take
an open covering $\{R_i: i\in I\}$ of $M$ such that $R_i \Subset S_i$.
As before, we choose a tubular neighbourhood of $M$ in $M'$ and a
(fibrewise) Riemannian metric on it, denoting by $U(\epsilon)$ the
open $\epsilon$-tube over the open subset $U\subset M$.  Then for each
$i\in I$, there exists $\epsilon_i>0$ such that:
\begin{itemize} 
\item $R_i(\epsilon_i) \subset T_i^\circ$; 
\item $R_{ij}(\epsilon_i) \subset P_{ij}$; 
\end{itemize} 
because there are only finitely many $j\in I$ such that $S_{ij}$ is
nonempty.  Now the homomorphisms $\Phi_i|_{R_i(\epsilon_i)}$ glue
together to define a global homomorphism $\Phi \colon \cB_1'|_P \to
\cB_2'|_P$ over $P = \bigcup_{i\in I} R_i(\epsilon_i)$ such that $\phi
= \iota^{-1} \Phi$.  The uniqueness of $\Phi$ as a germ is obvious.
This completes the proof of Theorem~\ref{thm:with_modules}.

\section{Global Unfolding of TEP Structures}
\label{sec:unfolding}

As an application of our results we now prove a global unfolding theorem for TEP~structures, by globalizing the reconstruction theorem for germs of TEP structures due to Hertling--Manin~\cite{Hertling--Manin}.  This global unfolding theorem has applications in mirror symmetry and the study of quantum cohomology~\cite[\S8.1.6]{CI:Fock_sheaf}.  TEP structures were introduced by Hertling~\cite{Hertling}; they are closely related to Dubrovin's notion of Frobenius manifold~\cite{Dubrovin:ICM,Dubrovin:Frobenius,Hertling--Manin}.

\begin{definition}[TEP structure]
  Let $M$ be a complex manifold.
  A \emph{TEP structure} $(\cF, \nabla, (\cdot,\cdot)_{\cF})$ 
  with base $M$ consists of 
  a locally free $\cO_{M\times \C}$-module $\cF$ of rank $N+1$, and 
  a meromorphic flat connection 
  \[
  \nabla \colon \cF \to 
  (\pi^*\Omega_{M}^1 \oplus \cO_{M\times \C} z^{-1}dz) 
  \otimes_{\cO_{M\times \C}} \cF(M\times \{0\}) 
  \]
  so that for $f\in \cO_{M\times \C}$, $s\in \cF$,
  and tangent vector fields $v_1,v_2\in \Theta_{M \times \C}$:
  \begin{align*}
    \nabla (f s) = df \otimes s + f \nabla s &&
    [\nabla_{v_1}, \nabla_{v_2}] = \nabla_{[v_1,v_2]}  
  \end{align*}
  together with a non-degenerate pairing 
  \[
  (\cdot,\cdot)_{\cF} \colon 
  (-)^* \cF \otimes_{\cO_{M\times \C}} \cF \to \cO_{M\times \C}
  \]
  which satisfies
  \begin{align*} 
    \begin{split} 
      ((-)^*s_1,s_2)_{\cF} 
      & = (-)^* ((-)^* s_2, s_1)_{\cF}   \\ 
      d ((-)^* s_1, s_2)_{\cF}  
      & = ((-)^* \nabla s_1, s_2)_{\cF} 
      + ((-)^*s_1,\nabla s_2)_{\cF}  
    \end{split} 
  \end{align*}
  for $s_1,s_2 \in \cF$. 
  Here $\cF(M\times \{0\})$ denotes the sheaf of 
  sections of $\cF$ with poles of order at most~1 
  along the divisor $M\times \{0\} \subset M \times \C$. 
\end{definition} 

\begin{definition}[Miniversality]
  Let $M$ be a complex manifold.  A \emph{TEP structure} $(\cF, \nabla, (\cdot,\cdot)_{\cF})$ 
  with base $M$ is called \emph{miniversal} if for each $y \in M$, the set
  \[
  \big\{ x \in \cF|_{(y,0)} : \text{the map $T_y M \to \cF|_{(y,0)}$, $v \mapsto (z \nabla_v x) |_{(y,0)}$ is an isomorphism} \big\}
  \]
  is non-empty in the fiber $\cF|_{(y,0)}$.
\end{definition}

\begin{theorem}[Global unfolding for TEP structures]
  Let $M$ be a complex manifold and $(\cF, \nabla, (\cdot,\cdot)_{\cF})$ a TEP structure with base $M$.  
Suppose that for each $y\in M$, there exists a section $\zeta$ of $\cF$ 
over a neighbourhood of $(y,0)\in M\times \C$ such that:
  \begin{description}
  \item[(IC)] the map $T_y M \to \cF|_{(y,0)}$ defined by $v \mapsto z\nabla_v \zeta|_{(y,0)}$ is injective;
    \item[(GC)] the fiber $\cF|_{(y,0)}$ is generated by iterated derivatives 
      \begin{align*} 
      (z^2\nabla_{\partial_z})^l z\nabla_{v_1} \cdots z 
     \nabla_{v_k}\zeta\big|_{(y,0)} && l\ge 0 
      \end{align*} 
      with respect to local vector fields $v_1,\ldots,v_k$ on $M$ near $y$ 
      and $z^2 \partial_z$. 
  \end{description}
  Then there exist a complex manifold $M'$, a miniversal TEP structure $(\cF', \nabla', (\cdot,\cdot)_{\cF'})$ with base $M'$, and a closed embedding $\iota \colon M \to M'$ such that:
  \[
  \iota^\star \Big(\cF', \nabla', (\cdot,\cdot)_{\cF'}\Big) = \Big(\cF, \nabla, (\cdot,\cdot)_{\cF}\Big)
  \]
  Furthermore the manifold-germ $M'$ and the TEP structure $(\cF', \nabla', (\cdot,\cdot)_{\cF'})$ are unique up to unique isomorphism in the sense of Theorem~\ref{thm:main_result} and~\ref{thm:with_modules}.
\end{theorem}

\begin{proof}
  Combine Theorems~\ref{thm:main_result} and~\ref{thm:with_modules} with the universal unfolding theorem for germs of TEP structures proved by Hertling--Manin~\cite[Theorem~2.5,~Lemma~3.2]{Hertling--Manin}.
\end{proof}

Analogous global unfolding theorems for TE structures~\cite{Hertling--Manin}, log-trTLEP structures~\cite{Reichelt}, and so on can be proved in exactly the same way.  Global unfoldings of log-trTLEP structures have interesting applications in Gromov--Witten theory~\cite{CI:modularity}.

\bibliographystyle{plain}
\bibliography{bibliography}

\begin{thebibliography}{10}

\bibitem{CI:Fock_sheaf}
Tom Coates and Hiroshi Iritani.
\newblock A {F}ock sheaf for {G}ivental quantization.
\newblock in preparation.

\bibitem{CI:modularity}
Tom Coates and Hiroshi Iritani.
\newblock Modularity of {G}romov--{W}itten potentials for local $\mathbb{P}^2$.
\newblock In preparation.

\bibitem{Dubrovin:Frobenius}
Boris Dubrovin.
\newblock Geometry of {$2$}{D} topological field theories.
\newblock In {\em Integrable systems and quantum groups ({M}ontecatini {T}erme,
  1993)}, volume 1620 of {\em Lecture Notes in Math.}, pages 120--348.
  Springer, Berlin, 1996.

\bibitem{Dubrovin:ICM}
Boris Dubrovin.
\newblock Geometry and analytic theory of {F}robenius manifolds.
\newblock In {\em Proceedings of the {I}nternational {C}ongress of
  {M}athematicians, {V}ol. {II} ({B}erlin, 1998)}, number Extra Vol. II, pages
  315--326, 1998.

\bibitem{FOOO}
Kenji Fukaya, Yong-Geun Oh, Hiroshi Ohta, and Kaoru Ono.
\newblock Technical details on {K}uranishi structure and virtual fundamental
  chain.
\newblock \href{http://arxiv.org/abs/1209.4410}{\texttt{arXiv:1209.4410
  [math.SG]}}, 2012.

\bibitem{Hertling}
Claus Hertling.
\newblock {$tt^*$} geometry, {F}robenius manifolds, their connections, and the
  construction for singularities.
\newblock {\em J. Reine Angew. Math.}, 555:77--161, 2003.

\bibitem{Hertling--Manin}
Claus Hertling and Yuri Manin.
\newblock Unfoldings of meromorphic connections and a construction of
  {F}robenius manifolds.
\newblock In {\em Frobenius manifolds}, Aspects Math., E36, pages 113--144.
  Vieweg, Wiesbaden, 2004.

\bibitem{Joyce}
Dominic Joyce.
\newblock D-manifolds and d-orbifolds: a theory of derived differential
  geometry.
\newblock Available at \url{http://people.maths.ox.ac.uk/joyce/dmbook.pdf},
  2012.

\bibitem{Kashiwara}
Masaki Kashiwara.
\newblock {\em {$D$}-modules and microlocal calculus}, volume 217 of {\em
  Translations of Mathematical Monographs}.
\newblock American Mathematical Society, Providence, RI, 2003.
\newblock Translated from the 2000 Japanese original by Mutsumi Saito, Iwanami
  Series in Modern Mathematics.

\bibitem{Reichelt}
Thomas Reichelt.
\newblock A construction of {F}robenius manifolds with logarithmic poles and
  applications.
\newblock {\em Comm. Math. Phys.}, 287(3):1145--1187, 2009.

\bibitem{Shiota}
Masahiro Shiota.
\newblock Some results on formal power series and differentiable functions.
\newblock {\em Publ. Res. Inst. Math. Sci.}, 12(1):49--53, 1967/77.

\bibitem{Whitney--Bruhat}
H.~Whitney and F.~Bruhat.
\newblock Quelques propri\'et\'es fondamentales des ensembles
  analytiques-r\'eels.
\newblock {\em Comment. Math. Helv.}, 33:132--160, 1959.

\end{thebibliography}

\end{document}